\documentclass[12pt,oneside]{amsart}
\usepackage{amsmath,amsfonts}
\usepackage{amsthm}
\usepackage{tikz}
\author{Youssef Rami and Younes Derfoufi}
\usepackage{geometry}
\usepackage{graphics}
\numberwithin{equation}{subsection}

\address{D\'epartement de Math\'ematiques \& Informatique, Universit\'e  My Ismail, B. P. 11 201 Zitoune, Mekn\`es, Morocco, Fax: (212) 5 35 53 68 08}
\email{y.rami@fs-umi.ac.ma ; youderf@gmail.com }

\usetikzlibrary{fit,calc,positioning,decorations.pathreplacing,matrix}

\newtheorem{theorem}{Theorem}

\newtheorem{corollary}[theorem]{Corollary}

\newtheorem{definition}[theorem]{Definition}
\newtheorem{example}[theorem]{Example}

\newtheorem{lemma}[theorem]{Lemma}
\newtheorem{notation}[theorem]{Notation}

\newtheorem{proposition}[theorem]{Proposition}
\newtheorem{remark}[theorem]{Remark}

 \usepackage{graphicx}

\begin{document}

\title{A variant of the Topological complexity of a map}
\maketitle

\begin{abstract}
In this paper, we associate to two given continuous maps $f,g: X\rightarrow Z$, on a path connected space $X$, the relative topological complexity   $TC^{(f, g, Z)}(X):=TC_X(X\times _ZX)$  of their fiber  space $X\times _ZX$. When $g=f$ we obtain a variant of the topological complexity $TC(f)$  of $f: X\longrightarrow Z$ generalizing Farber's topological complexity $TC(X)$ in the sens that $TC(X)=TC(cst_{x_0})$;  being $cst_{x_0}$ the constant map on $X$.  Moreover, we prove that $TC(f)$ is a fiberwise homotopy equivalence invariant. When $(X,x_0)$ is a pointed space, we prove that $TC^{(f, cst_{x_0}, Z)}(X)$ interpolates $cat(X)$ and $TC(X)$ for any continuous map $f$.
\end{abstract}
\section{Introduction}
Let $X$ be a path connected space and  $P(X)$
the set of all continuous paths $\gamma :[0,1]\longrightarrow X$ equipped with the compact open topology.
 The map  $\pi :P(X)\longrightarrow X\times X$ which assigns to any path  $\gamma $ its end points
$(\gamma (0),\gamma (1))$ is a  Serre  fibration called thereafter the path fibration of $X$.
A {\it motion planning algorithm} on $X$ is any global section of $\pi$, that is any function $s:X\times X\ \longrightarrow
PX$ $ $satisfying  $\pi \circ s=Id_{X\times X}$. By \cite{Far03}, such $s$ 
 is continuous   if and
only if $X$ is contractible. When $X$ stands for the configuration space of a physical system (i.e. a robot), any continuous global section of $\pi$, describes how the robot moves autonomously and continuously in $X$. In a general situation,
  M. Farber defined  in \cite{Far03} 
 the {\it topological complexity} $TC(X)$ of  $X$ to be the smallest integer k  for which
there is an open cover $\{U_{i}\}_{i=1}^{k}$\ of $X$ $\times X$\ and a set $\{s_{i}:U_{i}\longrightarrow PX\}_{i=1}^{k}$ of   continuous {\it local sections} 
satisfying $\pi \circ s_{i}(x)=x$  for any $x$ in
$U_{i}$. We  then write   $TC(X)=k$ and if there is no such integer k 
we put, by convention,  $TC(X)=\infty.$ Recall that $TC(X)$ is effectively the  Schwarz's genus of the path fibration $\pi$ \cite{Sch}. 

Now, let $A\subseteq X\times X$ and  $P_A(X)$ the subspace of $P(X)$ formed by all paths in $X$ whose end points are in $A$. The pullback of the inclusion  $A\hookrightarrow X\times X$ induces the fibration  $\pi _A : P_A(X)\longrightarrow A$. Its Schwarz's genus is by definition the {\it relative topological complexity} $TC_X(A)$ \cite{Far08}.

 The goal of this paper is to study the particular case where $A$ designates the fibred space $X\times _ZX = (f\times g)^{-1}(\Delta Z)$ of two continuous maps $f, g: X\longrightarrow Z$. $TC_X(A)$ will be denoted (in this case) by $TC^{(f,g,Z)}(X)$ and called {\it fibred topological complexity} of $f$ and $g$ over $Z$. 
 As a particular case, taking $f=g=c_{z_0}$ (for some fixed $z_0$ in $Z$) the constant map, we recover Farber's topological complexity $TC(X)=TC^{(c_{z_0}, c_{z_0}, Z)}(X)$.
   Wherefore, we put the following
\begin{definition}
Let $f: X\longrightarrow Z$ any map with source a path connected  space, the {\it topological complexity} of $f$ is by definition $TC(f)= TC^{(f, f, Z)}(X)$.
\end{definition} 
Notice that Petar Pavešić defined in \cite{Pav} the notion of topological complexity of a surjective map $f :X\rightarrow Y$ between path-connected spaces in terms of partial sections of  the projection map  $\pi _f: X^I \rightarrow X\times Y$ given
as
$$\pi _f(\alpha) := (Id_X \times f)(\pi (\alpha)) = (\alpha(0), f(\alpha(1))).$$ Although Pavešić's definition is analogous to that of $TC(X)$, it is not an homotopy invariant due to the fact that $\pi _f$ is a fibration if, and only if, $f$ is. However, He proves that it is invariant under  fiberwise homotopy equivalences, (i.e. an FHE-invariant for short). Thus, our definition below gives a variant for the topological complexity of $f$.

Now, let $X$  be a configuration space of a given physical system $\mathcal{S}$, that is each point of $X$ represents one stat of $\mathcal{S}$. Any map $f: X\longrightarrow Z$ can be viewed as an action that assigns to any stat of $\mathcal{S}$ a specific function. Therefore, $TC(f)$ measures the {\it complexity of motion planing algorithms}  between endpoints  to which the same function is assigned. As a particular case, if $G$ is a group acting on $X$ and $p: X\longrightarrow X/G$ stands for the usual projection, then $TC(p)$ measures the complexity of motion planing algorithms between  points $x$ and $gxg^{-1}$, with $x\in X$ and $g\in G$. When $G$ acts transitively on $X$, the quotient space $X/G$ reduces to a single point $\{\bar{e}\}$ so that $p$ becomes a constant map and then $TC(p)=TC(X)$.

 The first main result in this  paper  stats the $TC{(f)}$ is also an FHE-invariant: 
 \begin{theorem} (Corollary 7): Let $f : X\rightarrow Z$ be  a continuous map with source a path connected space and $v : X'\rightarrow X$ a homotopy equivalence, then $TC(f\circ v)=TC(f)$.
 \end{theorem}
 Our second main result is the 
 \begin{theorem}(Proposition 9): Let $f : (X,x_0)\rightarrow (X,x_0)$ be a continuous map with $(X,x_0)$ a pointed path connected space. Then 
$cat(X)\leq TC^{(f,cst_{x_0},X)}(X)\leq TC(X)$.
 \end{theorem}

\section{Fibered topological complexity} 
The following definition gives an explicit formulation for $TC^{(f,g,Z)}(X)$.
\begin{definition}
Let $%
f:X\longrightarrow Z$ and  $g:X\longrightarrow Z$ be two maps from a path connected space $X$. The fibred topological complexity   $TC^{(f,g,Z)}(X)$ is the smallest integer 
$k$ for which there is an open cover \ $\{U_{i}\}_{i=1}^{k}$\ of  $A=X\times 
_{Z}X$\ and a set 
 $\{s_{i}:U_{i}\longrightarrow P(X)\}_{i=1}^{k}$ of continuous maps such that  $s_i(x,y)(0)=x$ and $s_i(x,y)(1)=y$, for all $(x,y)\in U_i$. 
 If there is no such
integer $k$, we put : \ $TC^{(f,g,Z)}(X)=%
\infty .$
\end{definition}
Notice that each map  $s_{i}:U_{i}\longrightarrow P(X)$ in the above definition can be viewed as
a local section $s_{i}:U_{i}\longrightarrow P_A(X)$ (of $\pi _A: P_A(X)\longrightarrow A$), followed by the injection $P_A(X)\hookrightarrow P(X)$. For this reason, we will still call it a local section.
\subsection{Fibered Topological Complexity of fiber bundles}
In this subsection, we assume that all maps are fiber bundles on the 
base space $Z$. In order to avoid any confusion, the fiber product $X\times _{Z}X=\{(x,y)\in
X\times X\; \mid f(x)=g(y)\}$ will be denoted by $X\times _{Z}^{f,g}X$. Let $\ \varphi
:(X,f,Z)\longrightarrow (X^{\prime },f,Z)$ and $\varphi ^{\prime
}:(X,g,Z)\longrightarrow (X^{\prime \prime },g^{\prime },Z)$  two fiber bundle isomorphisms described by the following commutative diagrams:%
$$\begin{array}{ccccccc}
X & \stackrel{\varphi}{\longrightarrow} & X' \qquad  \qquad X & \stackrel{\varphi '}{\longrightarrow} & X" \\
f \searrow &  & \swarrow  f' \qquad  \qquad g \searrow &  & \swarrow  g" \\
& Z & \qquad \qquad & Z & 
\end{array}
$$
By putting $\ g^{\prime
}=g^{''}\circ \varphi ^{\prime }\circ \varphi ^{-1}$ we have $g^{\prime }\circ
\varphi =g$. Thus, we can replace 
 $X^{''}$ by $\ X^{\prime}$ to  obtain the following commutative diagrams :%
$$\begin{array}{ccccccc}
X & \stackrel{\varphi}{\longrightarrow} & X' \qquad  \qquad X & \stackrel{\varphi }{\longrightarrow} & X' \\
f \searrow &  & \swarrow  f' \qquad  \qquad g \searrow &  & \swarrow  g' \\
& Z & \qquad \qquad & Z & 
\end{array}
$$
\begin{theorem}
With the  data just above, we have $TC^{(f,g,Z)}(X)=TC^{(f^{\prime
},g^{\prime },Z)}(X^{\prime })$
\end{theorem}

\begin{proof}
Let $n=TC^{(f,g,Z)}(X)$. There exist then  an open cover $U_{1},\ldots , U_{k}$ of $%
X\times _{Z}^{f,g}X$ and $n$ locales sections $s_{i}:U_{i}\longrightarrow P(X)$
such that $s_{i}(x,y)(0)=x$ and $s_{i}(x,y)(1)=y$ for all $(x,y)\in U_{i}$.
The subspaces   $V_{i}=(\varphi \times \varphi )(U_{i})$ ($1\leq i\leq n$)
 form an open cover of $X^{\prime }\times
_{Z}^{f^{\prime },g^{\prime }}X^{\prime }.$  Now, the maps  $s_{i}^{\prime }:V_{i}\longrightarrow P(X')$ given by $%
s_{i}^{\prime }(x',y')(t)=\varphi \lbrack s_{i}(\varphi ^{-1}(x'),\varphi
^{-1}(y'))(t)]$ ($1\leq i\leq n$, $(x',y')\in V_i$) are clearly local sections associated to the open cover. Therefore $TC^{(f^{\prime },g^{\prime },Z)}(X^{\prime
})\leq TC^{(f,g,Z)}(X)$. By the same way, one proves that $TC^{(f^{\prime
},g^{\prime },Z)}(X^{\prime })\geq TC^{(f,g,Z)}(X)$.
\end{proof}
\subsection{The FHE-homotopy invariance} 
Let  $\ v:X'\longrightarrow X$   be an homotopy
equivalence between two path connected spaces $X$ and $X'$ and $u: X\rightarrow X'$ its inverse. The following theorem generalizes the above one and states the fiberwise  homotopy equivalence invariance of the fibred topological complexity:
\begin{theorem}
With the notations above, we have $TC^{(f,g,Z)}(X) = TC^{(f\circ v,g\circ v,Z)}(X^{\prime }).$
\end{theorem}
\begin{proof}
Recall that $TC^{(f,g,Z)}(X)= TC_X(A)$ where $A=X\times _ZX$. Hence, taking $f\circ v, g\circ v : X'\longrightarrow Z$, we have $A' = X'\times _ZX'=\{(x',y'),  \; f\circ v(x') = g\circ v(y')\}$.
Let $k=TC^{(f,g,Z)}(X)$. There is $k$ open subspaces   $U_{1},\ldots , U_{k}$ covering  $%
X\times _{Z}X$ and $k$ local continuous   sections $s_{i}:U_{i}\longrightarrow P(X)$ such that $%
s_{i}(x,y)(0)=x$ and $s_{i}(x,y)(1)=y$ for all $(x,y)\in U_{i}$  ($1\leq i\leq k$). Put $%
V_{i}=(v\times v)^{-1}(U_{i})=\{(x',y')\in X^{\prime }\times X^{\prime
}/(v(x'),v(y'))\in $ $U_{i}\}\subset $ $X^{\prime }\times _{Z}X^{\prime
}$.
$V_{1},\ldots ,V_{k}$ constitute obviously an open cover of $X^{\prime }\times
_{Z}X^{\prime }$. Consider then $$s"_i : V_i\stackrel{(v\times v)}{\longrightarrow} U_i \stackrel{s_i}{\longrightarrow}P(X)\stackrel{u^*}{\longrightarrow} P(X')$$
where $u^*(\gamma) = u\circ \gamma $ for all $\gamma \in P(X)$. Therefore,  for any $(x',y')\in V_i$ and any $t\in I$, we have $s"_i(x',y')(t)=u\circ s_i(v(x'),v(y'))(t)$, so that $s"_i(x',y')(0)=u\circ v(x')$ and $s"_i(x',y')(1)=u\circ v(y')$. Consider also $ H: X'\times I\longrightarrow X'$ the   homotopy between $u\circ v$ and $Id_{X'}$. Since $I$ is paracompact, the two spaces of continuous maps $\mathcal{C}^0(V_i, P(X'))$ and $\mathcal{C}^0(V_i\times I, X')$ are homeomorphic \cite{Dug}.
 Thus $H$ defines a unique continuous map $h: X'\longrightarrow P(X')$ such that $h(a')(t)=H(a',t)$ for all $a\in X'$ and $t\in I$. Finally, we consider $s'_i : V_i\longrightarrow P(X')$ the conciliation of the  paths $h(x')^{-1}$, $ s"(x',y')$ and $h(y') $ denoted as follows:
 $$s'_i(x',y')=h(x')^{-1}*s"_i(x',y')*h(y')$$
 where $h(x')^{-1}(t)=h(x')(1-t)$ for all $(x',y')\in V_i$ and all $t\in I$. Thus $s'_i$ is a continuous local section of $\pi _{A'}:  P_X(A')\longrightarrow A'$. More explicitly, $s'_i$ is given as follows:
 \begin{equation*}
 s'_{i}(x',y')(t)\left\{ 
 \begin{array}{c}
 H(x',3t)\text{ \ for \ }0\leq t\leq \frac{1}{3} \\ 
 u[s_{i}(v(x'),v(y'))(3t-1)]\text{ \ for \ }\frac{1}{3}\leq t\leq \frac{2}{3}
 \\ 
H(y',3t-1)\text{ \ for }\frac{2}{3}\leq t\leq 1%
 \end{array}%
 \right.
 \end{equation*}%
   Consequently, we have $TC^{(f,g,Z)}(X)\geq TC^{(f\circ v,g\circ v,Z)}(X')$.
   
    To obtain the other inequality,  let $V_{1},\ldots , V_{k}$ be  a covering  of $%
   X'\times _{Z}X'$ (with respect to $f\circ v,\; g\circ v: X'\longrightarrow Z$)   equipped with $k$ local continuous   sections $s'_{i}:V_{i}\longrightarrow P(X')$ such that $%
   s'_{i}(x',y')(0)=x'$ and $s'_{i}(x',y')(1)=y'$ for all $(x',y')\in V_{i}$ and consider $U_i =(u\times u)^{-1}(V_i)$ ($1\leq i\leq k$) and $ H': X\times I\longrightarrow X$ the   homotopy between $Id_{X}$ and $v\circ u$. The following equation:
   \begin{equation*}
   s _{i}(x,y)(t)\left\{ 
   \begin{array}{c}
   H'(x,3t)\text{ \ for \ }0\leq t\leq \frac{1}{3} \\ 
   v[s'_{i}(u(x),u(y))(3t-1)]\text{ \ for \ }\frac{1}{3}\leq t\leq \frac{2}{3}
   \\ 
   H'(y,3t-1)\text{ \ for }\frac{2}{3}\leq t\leq 1%
   \end{array}%
   \right.
   \end{equation*}%
  asserts that  $TC^{(f,g,Z)}(X)\leq TC^{(f\circ v,g\circ v,Z)}(X')$.   
\end{proof}
\begin{corollary}
Let $f: X\longrightarrow Z$  be a continuous map and $v: X'\longrightarrow X$ an homotopy equivalence. Then $TC(f)=TC(f\circ v)$. That is $TC(f)$ is a FHE invariant.
\end{corollary}
 Being a particular case of the relative topological complexity \cite{Far08},  $TC^{(f,g,Z)}(X)$ satisfies the following
\begin{proposition} With the notations above, we have
\begin{enumerate}
\item $TC^{(f,g,Z)}(X)\leq TC(X)$.
\item If $A=X\times _ZX\subseteq B\subseteq X\times X$ and the inclusion map $B\hookrightarrow X\times X$ is homotopic to a map $B\rightarrow X\times X$ whose values are in $A$, then $TC^{(f,g,Z)}(X)=TC_X(B)$. 
\end{enumerate}
\end{proposition}
\section{Fibred Topological complexity  and LS-category}

In this section, unless otherwise stated, we consider the particular case where $Z=X$ is a pointed space at $x_0\in X$, $f: X\longrightarrow X$ is a pointed map and  $g=cst_{x_0}: X%
\longrightarrow X$, the constant map.  Thus  $X\times
_{Z}X=\{(x,y)\in X\times X\; \mid \; f(x)=x_0\}=f^{-1}\{x_0\}\times X$. The following proposition gives an invariant interpolating $cat(X)$ and $TC(X)$:

\begin{proposition}
$cat(X)\leq TC^{(f,cst_{x_0},X)}(X)\leq TC(X)$
\end{proposition}

\begin{proof} We already know that $TC^{(f,f_{{x_0}},X)}(X)\leq TC(X).$ Now, let $%
n=TC^{(f,f_{{x_0}},X)}(X)$. Denote by  $U_{1},\ldots , U_{n}$ an open cover of $A=: f^{-1}\{{x_0}\}\times X$ and for each $1\leq i\leq n$;  $s_{i}:U_{i}\longrightarrow P_A(X)$  a
continuous section of $\pi _A:  P_X(A)\longrightarrow A$ so that $\pi _A\circ s_{i}=Id_{U_{i}}$. Let $\alpha
:X\longrightarrow X\times X$ \ be the map defined by  $\alpha (y)=
(x_0,y)$ and   $V_{i}=\alpha ^{-1}(U_{i})$ ($1\leq i\leq n$). 
Obviously, $
V_{1},\ldots , V_{k}$ \ form an open cover of $X$. Moreover, each  $%
h_{i}(y,t)=(s_{i}\circ \alpha (y))(t)$ is a homotopy between $cst_{x_0}$ and $Id_{V_i}$. Hence, $%
cat(X)\leq TC^{(f,f_{x_0},X)}(X)$.
\end{proof}
By \cite{Far03}, we have:
\begin{corollary}
If $G$ is a topological group, then $cat(G)=TC^{(f,cst_{e_G},G)}(G)=TC(G)$.
\end{corollary}
\subsection{Algorithms starting or ending on a fixed point $x_0$}
As a particular case of the proposition, by taking $f=Id_X$ we have $X\times _{Z}X=\{x_0\}\times X$, we obtain new interpretations of LS-category as follows: 
\begin{corollary}
If $X$ is a path connected  space , then   $$TC^{(Id_X,cst_{x_0},X)}(X)%
= cat(X)=TC^{(cst_{x_0},Id_X,X)}(X).$$
\end{corollary}
\begin{proof} Indeed, this is exactly the relation proved earlier in \cite[Lemma 4.23]{Far08} to which we recall the proof here. 
Let  $n=cat(X)$. There is an open cover $U_{1},\ldots ,U_{n}$ and for each $1\leq i\leq n$, a homotopy $h_{i}:U_{i}\times I\longrightarrow X$
satisfying: $h_{i}(x,0)=x_0$ and $h_{i}(x,1)=x$. 
Let  $\beta :\{x_0\}\times
X\longrightarrow X$ defined by $\beta (x_0,x)=x$. The subspaces   $V_{i}=\beta
^{-1}(U_{i})$ ($1\leq i\leq n$) obviously form an open cover of $\{x_0\}\times X$. Hence, each map $s_i$ defined on $V_i$ by $s_{i}(x_0,x)(t)=h_{i}(x,t)$   is a  section of the 
application $\pi :PX\longrightarrow V_{i}$ so that $TC^{(Id_X,cst_{x_0},X)}(X)%
\leq cat(X)$.
\end{proof}

\ 
\subsection{Loop free and based loop algorithms}

In this subsection we denote  by $LX=\{\gamma \in PX \; \mid \; \gamma (0)=\gamma (1)\}$ the space  of free loops in $X$. 
\begin{definition}
We call a fibered loop motion planning algorithm any continuous map $s_{LP}^{(f,g,Z)}:A=X\times _{Z}X\longrightarrow
LX$ satisfying : $s_{LP}^{(f,g,Z)}(x,y)(0)=s_{LP}^{(f,g,Z)}(x,y)(1)=x$
et $s_{LP}^{(f,g,Z)}(x,y)(\frac{1}{2})=y.$
\end{definition}
\begin{definition}
We define the fibered loop topological complexity,  the smallest integer $n$
denoted by $TC_{LP}^{(f,g,Z)}(X)$ for which there exist an open cover $%
U_{1}, \ldots ,U_{n}$ of $X\times _{Z}X$ and continuous sections $%
s_{i}:U_{i}\longrightarrow LX$ of $\pi _A: P(X)\rightarrow A$,
  satisfying $%
s_{i}(x,y)(0)=s_{i}(x,y)(1)=x$ and $s_{i}(x,y)(\frac{1}{2})=y.$ If there is
no integer n satisfying these conditions we put $TC_{LP}^{(f,g,Z)}(X)=\infty 
$
\end{definition}
\begin{proposition}
For any topological space $X$, we have : $TC_{LP}^{(f,g,Z)}(X)=TC^{(f,g,Z)}(X)$.
\end{proposition}
\begin{proof}
Similar to the proof of \cite[Theorem 2]{D-M}  where it is proven that $TC(X) = TC^{LP}(X)$.
\end{proof}

Next, we assume that $Z=X$ is a based space on $x_0\in X$, $g=cst_{x_0}$. We denote by $\Omega _{x_0}X=\{\gamma \in PX\; \mid \; \gamma
(0)=\gamma (1)=x_0\}$, the space of based loops in $X$. 
\begin{definition}
We call a fibered based loop  motion planning algorithm (rel. $f: X\rightarrow X$), any continuous map $%
s_{LP}^{(f,cst_{x_0},X)}:X\times _{X}X=\{x_0\}\times X\longrightarrow \Omega _{x_0}X$
such that  $s_{LP}^{(f,cst_{x_0},X)}(x_0,y)(0)=s_{LP}^{(f,cst_{x_0},X)}(x_0,y)(1)=x_0$
and $s_{LP}^{(f,cst_{x_0},X)}(x_0,y)(\frac{1}{2})=y$.
\end{definition}
\begin{definition}
We define a fibered  based loop topological complexity (rel. f) denoted by $TC_{LP}^{(f, cst_{x_0} ,X)}(X)$ to be  the smallest
integer $n$   for which there exist an
open cover $U_{1},\ldots ,U_{n}$ of $X\times _{X}X=f^{-1}(\{x_0\})\times X$ and 
continuous sections $s_{i}:U_{i}\longrightarrow \Omega _{x_0}X$ such that $%
s_{i}(x_0,y)(0)=s_{i}(x_0,y)(1)=x_0$ and $s_{i}(x_0,y)(\frac{1}{2})=y$.
\end{definition}
It follows from Corollary 9 and Proposition 14 that 
\begin{proposition}
If $X$ is  path connected based space, then $TC_{LP}^{(Id_X,cst_{x_0},X)}(X)=cat(X)=TC_{LP}^{(cst_{x_0}, Id_X, X)}(X)$.
\end{proposition}
\section{Some examples and remarks}
In this section, we give some examples and two remarks to illustrate the possible use  of the topological  complexity of a map.
\begin{example}
\begin{enumerate}
\item Considering the  $n^{-th}$-projective space  $\mathbb{R}P^n$ as the quotient of $S^n$ by $G=\mathbb{Z}_2$, we have the projection map $p : S^n\rightarrow \mathbb{R}P^n$. Following, for instance,  \cite[Example 4.1]{Far08}, we deduce that $TC(p)=TC_{S^n}(A)= 1\; \hbox{or} \; 2$  for $n$  odd or even respectively. Here, $A= \{(x,-x); x\in S^n\}\cup \Delta (S^n)$, being $\Delta (S^n)$ the diagonal of $S^n$.
\item Let $G=\mathbb{Z}_2$  acting on $X= T^2 =S^1\times S^1$ by putting $(-1)(z,\omega)=(\bar{z}=1/z, -\omega)$. The  quotient space is the  Klein bottle $K$. Thus, if $p : T^2\rightarrow K$ stands for the projection map, we have $TC(p)\leq TC(T^2)=3$. Now,
 $TC(p)=TC_X(A)$ where $A=\{((z,\omega),(\bar{z}, -\omega)); (z,\omega)\in T^2\}\cup \Delta (T^n)$. 
 It results that  $TC(p)$ measures the complexity of  motion planing algorithms between  pairs of points $(z,\omega)$ and $(\bar{z},-\omega)$ which are identified to form an orbit in $K=T^2/\mathbb{Z}_2$. 
\item Let $\iota : K\rightarrow \mathbb{R}^3$ be  an immersion of the Klein bottle in the $3$-dimensional space $\mathbb{R}^3$. By \cite{C-V}, we know that $TC(K)=4$ so that $TC(\iota)\leq 4$. In this case, $TC(\iota)=TC_K(A)$ where $A=\{(x,y)\in K\times K; \iota (x)=\iota (y)\}$ stands for  the preimage of  
 the part where $K$  auto-intersect. Thus $TC(\iota)$ measures   the complexity of  motion planing algorithms whose endpoints are sent in that part. 
 \end{enumerate}
 \end{example}
 \begin{remark}
 \begin{enumerate}
  \item 
  By \cite[Proposition 4.24]{Far08}, for any covering of $X\times X$ by locally compact subsets $A_1, A_2, \ldots , A_r$   we have $TC_X(A_i)\leq TC(X)\leq TC_X(A_1)+\ldots +TC_X(A_r)$ ($1\leq i\leq r$). Thus, one can make use of $r$ continuous maps 
  $f_i : X\rightarrow Z$  and put for each $1\leq i\leq r$,  $A_i =(f_i\times f_i)^{-1}(\Delta (Z))$ so that the inequality above becomes  $TC(f_i)\leq TC(X)\leq TC(f_1)+\ldots +TC(f_r)$.
  
  Thanks to the interpretation of $TC(f)$ given in the introduction, an appropriate choice of $Z$ is then imperative and will allow to determine the value of $TC(X)$ and also the covering of $X\times X $ which realizes it.
  \item Obviously, if $f : X\rightarrow Z$ is an injective map, then $TC(f)=TC_X(\Delta (X))=1$ \cite[Lemma 4. 21]{Far08}. Moreover, if $f_1, f_2 : X\rightarrow Z$ are maps such that $(f_1\times f_1)^{-1}(\Delta (Z))\subseteq  (f_2\times f_2)^{-1}(\Delta (Z))$, then $TC(f_1)\leq TC(f_2)$. Thus,  by choosing  a sequence $(f_n: X\longrightarrow Z)_{n\geq 0}$ of continuous maps such that $(TC(f_n))_{n\geq 1}$ is an  increasing sequence,  we then obtain an approximation from  below: ${\lim}_{n\rightarrow \infty}  TC(f_n)\leq TC(X)=TC(cst_{z_0})$. Clearly, the choice of the sequence  $(f_n)$ depends on the "degrees" of non-injectivity of its terms.
 \end{enumerate}
 \end{remark}

\end{document}